\documentclass[11pt]{article}
\usepackage{amsmath}
\usepackage{mathrsfs}
\usepackage{amsfonts}
\usepackage{amsmath,amsthm}
\usepackage{amsmath,amssymb,amsthm,latexsym}
\usepackage{graphics}
\usepackage{subfigure}
\usepackage{float}
\usepackage{graphicx}
\usepackage{amscd}
\usepackage[all]{xy}
\usepackage{hyperref}

\newtheorem{theorem}{Theorem}[section]
\newtheorem{proposition}[theorem]{Proposition}

\newtheorem{definition}[theorem]{Definition}

\newtheorem{lemma}[theorem]{Lemma}

\theoremstyle{remark}
\newtheorem{remark}[theorem]{Remark}

\def\vec{\mathbf}
\def\<{\langle}
\def\>{\rangle}

\begin{document}
\title{\bf{Closed strong spacelike curves, Fenchel theorem and Plateau problem in the 3-dimensional Minkowski space  }}

\author{Nan Ye%
  \thanks{School of Mathematical Sciences, Peking University, Beijing 100871, People's Republic of China. \texttt{yen@pku.edu.cn}. }
\and Xiang Ma%
  \thanks{LMAM, School of Mathematical Sciences, Peking University,
 Beijing 100871, People's Republic of China. \texttt{maxiang@math.pku.edu.cn}, Fax:+86-010-62751801. Corresponding author. Supported by NSFC 11471021.}
  }

\date{\today}
\maketitle

\begin{center}
{\bf Abstract}
\end{center}

We generalize the Fenchel theorem for strong spacelike closed curves of index $1$ in the 3-dimensional Minkowski space, showing that the total curvature must be less than or equal to $2\pi$. Here strong spacelike  means that the tangent vector and the curvature vector span a spacelike 2-plane at each point of the curve $\gamma$ under consideration. The assumption of index 1 is equivalent to saying that $\gamma$ winds around some timelike axis with winding number 1. We prove this reversed Fenchel-type inequality by constructing a ruled spacelike surface with the given curve as boundary and applying the Gauss-Bonnet formula. As a by-product, this shows the existence of a maximal surface with $\gamma$ as boundary.

\hspace{2mm}


{\bf MSC(2000):\hspace{2mm} 52A40, 53C42, 53C50}

\section{Introduction}

To study the global properties of closed curves, an interesting idea is to associate a specific surface $M$ with the given curve $\gamma$ as boundary. Then we can control the geometry of $\gamma$ by the information of $M$, and vice versa.

As an illustration, let us consider a closed smooth space curve $\gamma$ in $\mathbb{R}^n$ ($n\ge 3$) which is assumed to bound a minimal disk $M$. The Gauss-Bonnet formula says
\begin{equation}\label{eq-gaussbonnet}
\int_{M}K\mathrm{d}M
+\int_{\partial M}\kappa_g \mathrm{d}s=2\pi,
\end{equation}
where $K$ is the Gauss curvature of $M$, $\mathrm{d}M$ is the area element with respect to the induced metric, $\kappa_g$ is the geodesic curvature of the curve $\partial M=\gamma\subset M$, and $s$ is the arc-length parameter. For such a Euclidean minimal surface it is well-known that $K\le 0$. There also holds $\kappa(p)\ge \kappa_g(p)=\kappa(p)\cdot\cos\theta_p$ at any $p\in \gamma$, where $\theta_p$ is the angle between the tangent plane of $M$ and the osculating plane of $\gamma$ at $p$. Combining with these two facts, we immediately obtain the conclusion of the Fenchel theorem \cite{Fenchel}:
\[\int_\gamma \kappa\mathrm{d}s\geq 2\pi.\]
The equality is attained exactly when $M$ is flat and $\gamma$ is a convex plane curve.

Notice that the total curvature gives a quantitative measure of the complexity of the space curve $\gamma$. A natural expectation is that when $\int_\gamma \kappa\mathrm{d}s$ is small, $\gamma$ should be simple, as confirmed by the Fary-Milnor theorem \cite{Milnor} that when $n=3$ and $\int_\gamma \kappa\mathrm{d}s\le 4\pi$, $\gamma$ is always a trivial knot.

The next reasonable guess is that the solution $M$ of the corresponding Plateau problem should also be nice under similar conditions. Consider  a minimal surface $M\subset\mathbb{R}^n$ of arbitrary topological type with boundary $\gamma=\partial M$. If $\int_\gamma \kappa\mathrm{d}s\le 4\pi$,
a remarkable theorem obtained in 2002 by Tobias Ekholm, Brian White and Daniel Wienholtz \cite{Ekholm} asserts that $M$ must be smoothly embedded.

In this paper, we are motivated to consider a similar picture, namely a closed spacelike curve $\gamma$ as the boundary of a spacelike (maximal) surface in the 3-dimensional Minkowski space $\mathbb{R}^3_1$. We would like to find some appropriate assumptions on $\gamma$ to guarantee that the Plateau problem has a solution. It is also desirable to give a prior estimation of the total curvature $\int_\gamma \kappa\mathrm{d}s$ (i.e., a generalization of the Fenchel theorem and/or the Fary-Milnor theorem). We achieve these goals successfully.
To state our result, let us introduce two definitions.

\begin{definition}\label{def-strong}
A curve $\gamma\subset \mathbb{R}^3_1$ is called spacelike if at any point the tangent vector is spacelike, i.e. $\langle \gamma',\gamma'\rangle>0$. It is called \emph{strong spacelike } if its (unit) tangent vector $T=\gamma'(s)$ and the curvature vector $\kappa N=\gamma''(s)$ span a spacelike 2-plane at each point. In other words, the osculating plane at any point of $\gamma$ is of rank-2 with a positive-definite inner product induced from $\mathbb{R}^3_1$.
\end{definition}

Note that in the 3-dimensional Minkowski space, there exist neither closed timelike curves, nor closed spacelike curves with timelike normals. In contrast, the strong spacelike condition allows the length and curvature to be defined directly as before and admits closed examples. To avoid misunderstanding, we point out that a strong spacelike curve does not allow inflection points where the curvature vector is a zero vector. So we can assume $\kappa>0$.

\begin{definition}\label{def-ind}
The \emph{index} of a closed spacelike curve in $\mathbb{R}^3_1$ is defined to be the winding number $I$ of the tangent indicatrix (the image of the unit tangent vector $T$) around the de Sitter sphere
$\mathbb{S}^2_1=\{X\in \mathbb{R}^3_1|\langle X,X\rangle=1\}$ (the usual one-sheet hyperboloid, which is homotopy equivalent to a circle). This index $I$ is integer-valued and always assumed to be positive.
\end{definition}

\begin{remark}
Note that this definition is independent to the choice of the timelike direction, hence well-defined. In contrast, for a closed curve in $\mathbb{R}^3$ generally there is not a well-defined notion of index or winding number unless it is a plane curve. In the special case that $I=1$, the closed curve winds around some timelike axis exactly for one cycle.
\end{remark}

Now we can state out generalization of the Fenchel theorem in the 3-dimensional Minkowski space. This seems to be a new result to the best of our knowledge.

\begin{theorem}[\textbf{Main Result 1; the Fenchel theorem in $\mathbb{R}^3_1$}]
\label{thm-fenchel}
Let $\gamma$ be a closed strong spacelike curve in $\mathbb{R}^3_1$ with index $1$. Then the total curvature $\int_\gamma k\mathrm{d}s\leq 2\pi$. The equality holds if and only if it is a convex curve on a spacelike plane.
\end{theorem}

\begin{remark}
The reversed inequality might look peculiar when compared with the Euclidean case. This can be explained as below. The total curvature of $\gamma$ is equivalent to the length of the tangent indicatrix $T(\gamma)$. If we consider $\gamma$ as a small perturbation of a closed convex plane curve, then the corresponding variation of $T(\gamma)$ in $\mathbb{S}^2_1$ is always along the \emph{timelike} co-normal direction and vibrates up and down along the equator, which makes the length $L(T(\gamma))$ less than the original length $2\pi$ of the equator.
On the other hand, although a line of altitude $\Gamma$ may have length greater than $2\pi$, it can not be realized as the tangent indicatrix $T(\gamma)$ of a closed strong spacelike curve $\gamma$, hence not a counterexample to our claim. This is because $\Gamma$ always lies in a half space, and after integration one gets a curve $\gamma$ whose height function (with respect to a fixed timelike direction) increases monotonically, thus can not be closed.
\end{remark}

To prove this reversed Fenchel inequality, we adopt the same idea that constructing a surface $M$ with $\partial M=\gamma$. Can we take $M$ to be a spacelike surface with vanishing mean curvature (called \emph{maximal surface})? A known criterion for the existence of solutions to this Plateau problem is as below.

\begin{theorem}\label{thm-Bartnik-Simon}\cite{Bartnik-Simon, Flaherty}
Given a compact, codimension two spacelike submanifold $\gamma^{n-2}$ in $\mathbb{R}^n_1$ without boundary. Suppose there is a spacelike hypersurface $M^{n-1}$ with $\partial M^{n-1}=\gamma^{n-2}$ and $M^{n-1}$ is the graph of a $C^2$ function $u$ defined over a compact domain $\Omega$ of spacelike subspace $\mathbb{R}^{n-1}$ whose gradient is uniformly bounded, $|Du|<\delta<1$.
Then there exists a spacelike maximal hypersurface $\overline{M}^{n-1}$ with $\partial \overline{M}^{n-1}=\gamma^{n-2}$ and $\overline{M}^{n-1}$ is also a graph over the same $\Omega$.
\end{theorem}

Thanks to this criterion, we need only to show the existence of a spacelike surface spanning $\gamma$. Fortunately, the assumption of $\gamma$ being strong spacelike with index $1$ ensures that $\gamma$ has a simple shape: its projection to a spacelike plane $\mathbb{R}^2$ must be a convex plane curve bounding a compact convex domain $\Omega$ (see Lemma~\ref{lem-project}); and any three points on $\gamma$ span a spacelike plane (see Lemma~\ref{lem-section}). This enables us to show the existence of such a surface $M$ by explicit construction.

\begin{proposition}\label{prop-ruled}
Let $\gamma$ be a closed strong spacelike $C^2$ (twice continuously differentiable) curve in $\mathbb{R}^3_1$ of index $1$. Then there exists a surface $M$, $\partial M=\gamma$, with the following properties:
\begin{enumerate}
\item $M$ is a \emph{ruled surface};
\item $M$ is a graph over a convex domain $\Omega$ on a spacelike plane, hence a topological disk;
\item $M$ is $C^2$-smooth and spacelike (including the boundary points);
\item $M$ has non-negative Gauss curvature.
\end{enumerate}
\end{proposition}

Compared with the Euclidean case, ruled surfaces in $\mathbb{R}^3_1$ are still saddle shaped, yet with non-negative Gauss curvature ($K\ge 0$) \cite{Izumiya}; and at the boundary $\gamma$, $\kappa_g(p)=\kappa(p)\cosh(\theta_p)\ge \kappa(p)$ where $\theta_p\in \mathbb{R}$ is the so-called hyperbolic angle between the tangent plane of $M$ and the osculating plane of $\gamma$ at $p$. Applying the Gauss-Bonnet formula to $M$ and using the similar argument as in $\mathbb{R}^3$, we obtain the generalized Fenchel theorem~\ref{thm-fenchel} immediately.

Moreover, based on Theorem~\ref{thm-Bartnik-Simon}, we can now confirm the existence of a solution to the Plateau problem.

\begin{theorem}[\textbf{Main Result 2}]\label{thm-maximal}
Let $\gamma$ be a closed, strong spacelike curve in $\mathbb{R}^3_1$ with index $I=1$. Then there exists a maximal surface $\overline{M}$ with $\partial \overline{M}=\gamma$. This $\overline{M}$ is a graph over a compact, convex domain $\Omega\subset \mathbb{R}^2$, thus itself is an embedded topological disk.
\end{theorem}

We believe that such an $\overline{M}$ should be unique. The proof might resemble that of the classical Rad\'{o}'s theorem \cite{Rado2} as below.

\begin{theorem}\label{thm-rado}\cite{Rado2}
If $\gamma\subset \mathbb{R}^n$ has a one-to-one projection onto the boundary of a convex
planar region $\Omega$, then any minimal disk bounded by $\gamma$ is the graph of a smooth
function over $\Omega$. In particular, it is smoothly embedded. If in addition $n=3$, then there is only one disk and there are no minimal varieties of other topological types.
\end{theorem}

\begin{remark}
Another candidate for a saddle-shaped surface $\tilde{M}$ is the graph of a harmonic function on the complex plane. Let $\Omega\subset\mathbb{R}^2$ be the compact convex domain bounded by the projection image of $\gamma$. Precisely speaking, we need to solve the Dirichlet problem for a harmonic function $u$ defined over $\Omega$, with the boundary value assigned by the height function of $\gamma$ (i.e., $\gamma=\{(t,u(t))|t\in\partial\Omega\}$). The existence of a solution $u$ is no doubt. Yet the non-trivial part is to guarantee that the graph $\tilde{M}=\{(z,u(z))|z\in \Omega\}$ is spacelike, i.e., to prove a gradient estimate $|Du|<1$. This depends on the boundary value determined by $\gamma$ and the assumption that $\gamma$ is strong spacelike. It seems that we can prove the desired estimation when $\Omega$ is a circular disk. But this method is not successful in the general case.
\end{remark}

We would like to mention that we have found several different proofs to the reversed Fenchel inequality (Theorem~\ref{thm-fenchel}). One of them uses a generalized Crofton formula \cite{MaYeWang}. We are also considering a generalization to higher dimensional spacelike submanifolds, and the results will appear elsewhere.\\

\textbf{Acknowledgement.}~
We would like to thank Donghao Wang for helpful discussions.

\section{Basic properties of strong spacelike closed curve of index $1$}

The 3-dimensional Minkowski space $\mathbb{R}^3_1$ is endowed with a Lorentz inner product, expressed in a canonical coordinate system as below:
\[\langle X,Y\rangle=x_1y_1+x_2y_2-x_3y_3, ~~~X=(x_1,x_2,x_3), Y=(y_1,y_2,y_3).\]
A vector $X$ is called spacelike (lightlike, timelike) if $\langle X,X\rangle>0 (=0,<0)$, respectively. A timelike vector $X=(x_1,x_2,x_3)$ is called \emph{future-directed} (\emph{past-directed}) if $x_3>0$ ($x_3<0$).

A 2-dimensional subspace $V$ is called spacelike (lightlike, timelike) if the Lorentz inner product restricts to be a positive definite (degenerate, Lorentz) quadratic form on $V$. It is the orthogonal complement of a nonzero vector $\vec{n}$ which is timelike (lightlike, spacelike) respectively. In particular we can define the \emph{cross product} of two spacelike vectors on a spacelike plane as below, which is always orthogonal to $X$ and $Y$ (when it is nonzero we obtain a timelike normal vector):
\begin{equation}
X\times Y=(-x_2 y_3+x_3 y_2, -x_3 y_1+x_1 y_3, x_1 y_2-x_2 y_1).
\end{equation}
A curve or a surface in $\mathbb{R}^3_1$ is said to be spacelike (lightlike, timelike) if its tangent space at each point is so.

The strong spacelike condition for a curve $\gamma$ is defined in the introduction (Definition~\ref{def-strong}). A basic way to visualize the corresponding geometric shape of $\gamma$ is projecting $\gamma$ to a spacelike or lightlike subspace $\Sigma\subset\mathbb{R}^3_1$.

\begin{definition}\label{def-project}
When $\Sigma$ is spacelike with unit timelike normal vector $\vec{n}$, the projection of any vector $\vec{v}\in \mathbb{R}^3_1$ to $\Sigma$ is \[
\sigma(\vec{v})\triangleq\vec{v}+\langle\vec{v},\vec{n}\rangle\vec{n}.
\]
When $\Sigma$ is a lightlike plane with lightlike normal $\vec{n}$, one should take care that $\vec{n}\in \Sigma$. Take an arbitrary lightlike vector $\vec{n}^\star$ so that $\langle \vec{n},\vec{n}^\star\rangle=1$. Note that $\vec{n}^\star$ is transversal to $\Sigma$, and it is not unique.
The projection of any vector $\vec{v}\in \mathbb{R}^3_1$ to $\Sigma$ is
defined as below which depends on both $\vec{n}$ and $\vec{n}^\star$:
\begin{equation}\label{eq-project}
\sigma(\vec{v})\triangleq
\vec{v}-\langle\vec{v},\vec{n}\rangle\vec{n}^\star.
\end{equation}
\end{definition}

\begin{lemma}\label{lem-project}[\textbf{The Projection Lemma}]\\
Let $\gamma$ be a closed strong spacelike curve in $\mathbb{R}^3_1$ with $I=1$.
Let $\sigma$ be the projection map to a spacelike or lightlike plane $\Sigma$ in $\mathbb{R}^3_1$.  Then $\sigma(\gamma)$ is a strictly convex Jordan curve on $\Sigma$, and $\sigma$ is a one-to-one correspondence.
\end{lemma}

\begin{proof}
Let $s$ be the arclength parameter of $\gamma$. In the canonical coordinate system, the tangent vector of $\gamma$ can be expressed in terms of the longitude and latitude parameters $\theta,\phi$:
\begin{equation}\label{eq-tangent}
T(s)=(\cosh\phi(s)\cos\theta(s),\cosh\phi(s)\sin\theta(s),
\sinh\phi(s)).
\end{equation}
The strong spacelike and $I=1$ assumption implies
\begin{equation}\label{eq-strong}
\cosh^2\phi\cdot \theta'(s)^2-\phi'(s)^2>0,
\end{equation}
and $\theta(s)$ ranges from $0$ to $2\pi$ monotonically.

When $\Sigma$ is a spacelike plane, without loss of generality we may take its normal vector $\vec{n}=(0,0,1)$. It follows from \eqref{eq-tangent} that the projection $\sigma(\gamma)$ has tangent vector
\begin{equation}
\frac{d}{ds}(\sigma(\gamma))
=\sigma(T)=(\cosh\phi(s)\cos\theta(s),\cosh\phi(s)\sin\theta(s)).
\end{equation}
So the tangent direction of $\sigma(\gamma)$ is the same as $(\cos\theta(s),\sin\theta(s))$, which rotates in a strictly monotonic manner with range $[0,2\pi]$. Thus $\sigma(\gamma)$ must still be a closed and strictly convex curve on $\Sigma$. The 1-1 correspondence property is clear.

When $\Sigma$ is a lightlike plane, without loss of generality we suppose it is orthogonal to $\vec{n}=(0,1,1)$ and transverse to $\vec{n}^\star=(0,\frac{1}{2},-\frac{1}{2})$. Let $\vec{e}_1=(1,0,0)$. By Definition~\ref{def-project} and \eqref{eq-project}, the projection image $\sigma(\gamma)$ has tangent vector
\begin{equation}
\frac{d}{ds}(\sigma(\gamma))=\sigma(T)
=\cosh\phi(s)\cos\theta(s)\vec{e}_1
+\frac{\cosh\phi(s)\sin\theta(s)+\sinh\phi(s)}{2}\vec{n}.
\end{equation}
The curvature being positive or not is an affine invariant property. So we need only to identify $\sigma(T)$ with the tuple $\tilde{T}(s)=(\cosh\phi(s)\cos\theta(s),\cosh\phi(s)\sin\theta(s)
+\sinh\phi(s))$ and to show that
\[
\det(\tilde{T},\tilde{T}')=(\theta'\cosh^2\phi+\phi'\cos\theta
+\theta'\cosh\phi\sinh\phi\cdot\sin\theta)
\] has a fixed sign. Indeed this can be shown using the Cauch-Schwarz inequality and the strong spacelike property \eqref{eq-strong} as below:
\begin{equation}
|\phi'\cos\theta+\theta'\sinh\phi\cosh\phi\sin\theta|
\leq\sqrt{\phi'^2+\theta'^2\sinh^2\phi\cosh^2\phi}
<\theta'\cosh^2\phi.
\end{equation}
The proof is thus completed.
\end{proof}

\begin{lemma}\label{lem-section}[\textbf{The Section Lemma}]\\
Assumptions as above. Let $p_1,p_2,p_3$ be arbitrary chosen distinct points on $\gamma$. Then
\begin{enumerate}
\item The line segment $\overline{p_1p_2}$ connecting $p_1,p_2$ is spacelike.
\item $p_1,p_2,p_3$ span a spacelike plane.
\item The chord $p_1p_2$ and the tangent line at $p_1$ span a spacelike plane.
\end{enumerate}
\end{lemma}

\begin{proof}
We prove (2) by contradiction first, and (1) follows as a corollary. Suppose that $\{p_1,p_2,p_3\}$ are contained in a timelike plane (this includes the collinear case). Then we can always find a spacelike plane $\Sigma$ so that the orthogonal projection of the plane is a line on $\Sigma$ . In particular, the projection of $\{p_1,p_2,p_3\}$ are collinear. On the other hand, the conclusion of Lemma~\ref{lem-project} says that these three points on a convex Jordan curve should always be distinct and not lying on a line. This is a contradiction. The situation that $\{p_1,p_2,p_3\}$ span a lightlike plane can be ruled out in a similar way (if only one be careful about the definition of the projection map).

For conclusion (3), suppose otherwise that $\overline{p_1p_2}$ and the tangent line at $p_1$ is contained in a timelike (or lightlike) plane. We can still choose a spacelike plane $\Sigma$ orthogonal to the timelike plane (or transversal to the lightlike plane); then the projection of $\overline{p_1p_2}$ and the tangent line at $p_1$ are collinear, which contradicts with Lemma~\ref{lem-project}, again.
\end{proof}

\section{The ruled spacelike surface spanning $\gamma$}

As we pointed out in the introduction, the proof of the two main results can be reduced to showing the existence of a saddle-shaped (hence $K\ge 0$) spacelike surface $M$ in $\mathbb{R}^3_1$ with the given boundary.
We tried several different saddle-shaped surfaces (like the graph of a harmonic function, or a minimal surface in $\mathbb{R}^3$ identified with this $\mathbb{R}^3_1$), with the same difficulty: How to show the constructed surface to be really \emph{spacelike} in this ambient Minkowski space?

After many unsuccessful attempts, we noticed that the Section Lemma~\ref{lem-section} seems to be helpful: it guarantees that we can take finitely many points $\{p_i\}$ on $\gamma$ and construct a polyhedron with spacelike triangular faces, which  would serve as an approximation to the expected \emph{smooth} spacelike surface. The obvious problem with this idea is that when we take more and more of the vertices $\{p_i\}$, we will not get a nice refinement of the triangulation in the usual sense; instead we would obtain something like a union of slim rulers. We were stuck here for some time until one day, the inspiration came that the limit shape might still exist as a smooth surface, which must be a ruled surface. Realizing that ruled surface is automatically saddle-shaped, the rest thing to do was then clear:
\begin{description}
  \item[Step 1:] Construct a ruled surface $M$ directly with $\partial M=\gamma$.
  \item[Step 2:] Show that $M$ is differentiable (using a suitable parametrization).
  \item[Step 3:] Show that $M$ is immersed and spacelike (with the help of the Section Lemma~\ref{lem-section}).
\end{description}

The rest of this section is devoted to the proof of Proposition~\ref{prop-ruled} following these three steps.\\

As a preparation, fix the canonical coordinate system $(x_1,x_2,x_3)$ in $\mathbb{R}^3_1$.
Choose the coordinate plane $Ox_1x_2$ as the target spacelike plane $\Sigma$ to make projection. \\

\textbf{Step 1}: First we give
an explicit construction of the ruled surface $M$ together with its parametrization.
Take two points $p$ and $q$ on $\gamma$ which divide $\gamma$ into two arcs with equal length $L$. Denote these two arcs as $\gamma_0(s),\gamma_1(s)$ with arc-length parameter $s\in [0,L]$ and tangent vectors $\gamma_0'(s)=T_0(s),\gamma_1'(s)=T_1(s)$, respectively. The ruled surface $X(s,t)$ is given as below:
\begin{equation}\label{eq-X}
X(s,t)=(1-t)\gamma_0(s)+t\gamma_1(s)
=X(s,\frac{1}{2})+\Big(t-\frac{1}{2}\Big)\vec{v}(s),
\end{equation}
\[
\text{where}~~X(s,\frac{1}{2})=\frac{1}{2}[\gamma_0(s)+\gamma_1(s)],~
\vec{v}(s)\triangleq\gamma_1(s)-\gamma_0(s),~~~
s\in[0,L],t\in[0,1].
\]

\textbf{Step 2}: $X(s,\frac{1}{2})$ and
$\vec{v}(s)$ are obviously $C^2$ curves because $\gamma$ is so. We have
\begin{eqnarray}
\frac{\partial X}{\partial t}&=&
\gamma_1(s)-\gamma_0(s)
=\vec{v}(s),\label{eq-X1}\\
\frac{\partial X}{\partial s}&=&
(1-t)T_0(s)+tT_1(s).\label{eq-X2}
\end{eqnarray}
In particular, $X(s,t)$ is a $C^2$ ruled surface even on a larger
domain $s\in [0,L], t\in (-\epsilon, 1+\epsilon)$ for some $\epsilon>0$.

\textbf{Step 3}: Finally we shall verify that $X(s,t)$ is a spacelike surface including the boundary points.

At the end point $p$ with $s=0$ we shall use the expansion
\[
\gamma_0(s)=p+T_0(0) s+ \kappa(0)N_0(0) \frac{s^2}{2} +o(s^2),~~
\gamma_1(s)=p+T_1(0) s+ \kappa(0)N_1(0) \frac{s^2}{2} +o(s^2),
\]
where $T_0(0)=-T_1(0)$ is the tangent vector of $\gamma$ at $p$,
$N_0(0)=N_1(0)$ is the normal vector, and $\kappa(0)>0$ is the curvature
at the same point. This implies that $u=s^2$ is a regular parameter
for the curve
\[
X(s,\frac{1}{2})=p+\kappa(0)N_0(0) \frac{s^2}{2}+o(s^2)
= p+\kappa(0)N_0(0) \frac{u}{2}+o(u)
\]
with tangent vector $ \frac{\kappa(0)N_0(0)}{2}\ne \vec{0}$. Thus
the tangent plane of $X(s,t)$ at $s=0$ is clearly
spanned by $N_0(0),T_0(0)$, hence it is the spacelike osculating plane
of the strong spacelike $\gamma$ at $p$.

At a generic boundary point, say $s\in (0,L)$ and $t=0$, the tangent plane spanned by $T_0(s),\vec{v}(s)$
is spacelike according to the conclusion (3) of the Section Lemma~\ref{lem-section}.
The argument is the same when $s\in (0,L), t=1$.

Moreover, we can choose the orientation suitably so that
$T_0(s)\times\vec{v}(s)$ and $T_1(s)\times\vec{v}(s)$ are both future-directed
timelike normal vector, because Lemma~\ref{lem-project} tells us
that after projection to $Ox_1x_2$-plane, $T_0(s),T_1(s)$ will point to
the same side of $\vec{v}(s)$. Since all future-directed
timelike vectors form the positive lightcone which is convex, we know
$(1-t)T_0(s)\times\vec{v}(s)+tT_1(s)\times\vec{v}(s)$ is also a future-directed timelike vector. This is nothing else but exactly the normal vector of
$X(s,t)$ at an interior point with $s\in(0,L)$ and $t\in(0,1)$, because
by \eqref{eq-X1} and \eqref{eq-X2} one finds
\[
\frac{\partial X}{\partial s}\times \frac{\partial X}{\partial t}=
(1-t)[T_0(s)\times\vec{v}(s)]+t[T_1(s)\times\vec{v}(s)].
\]
Thus $X(s,t)$ is spacelike and immersed at any interior point.

Finally, it is easy to see that a spacelike ruled surface must have non-negative
Gauss curvature in the Minkowski space (for example, see \cite{Izumiya}).
This finishes the proof to Proposition~\ref{prop-ruled}, and establishes Theorem~\ref{thm-fenchel}, Theorem~\ref{thm-maximal}.

\end{document}